\documentclass{amsart}

\usepackage{graphicx}

\usepackage{amssymb}
\usepackage{amsmath}

\newcommand{\erase}[1]{}
\theoremstyle{plain}
\newtheorem{theorem}{Theorem}[section]
\newtheorem{lemma}[theorem]{Lemma}

\newtheorem{corollary}[theorem]{Corollary}

\theoremstyle{definition}
\newtheorem{definition}[theorem]{Definition}

\newtheorem{claim}[theorem]{Claim}
\newtheorem{conjecture}[theorem]{Conjecture}

\theoremstyle{remark}
\newtheorem{remark}[theorem]{Remark}

\numberwithin{equation}{section}
\numberwithin{table}{section}

\numberwithin{equation}{section}
\numberwithin{table}{section}
\numberwithin{figure}{section}
\renewcommand{\qed}{\hfill {$\Box$}}

\newcommand{\C}{\mathord{\mathbb C}}
\newcommand{\F}{\mathord{\mathbb F}}
\renewcommand{\P}{\mathord{\mathbb  P}}
\newcommand{\Q}{\mathord{\mathbb  Q}}
\newcommand{\R}{\mathord{\mathbb R}}
\newcommand{\Z}{\mathord{\mathbb Z}}

\newcommand{\BBB}{\mathord{\mathcal B}}
\newcommand{\JJJ}{\mathord{\mathcal J}}
\newcommand{\KKK}{\mathord{\mathcal K}}
\newcommand{\LLL}{\mathord{\mathcal L}}
\newcommand{\MMM}{\mathord{\mathcal M}}
\newcommand{\UUU}{\mathord{\mathcal U}}
\newcommand{\WWW}{\mathord{\mathcal W}}

\newcommand{\maprightsp}[1]{\overset{#1}{\longrightarrow} }
\newcommand{\inj}{\hookrightarrow}
\newcommand{\surj}{\mathbin{\to \hskip -7pt \to}}
\newcommand{\isom}{\overset{\sim}{\to}}

\newcommand{\set}[2]{\{\; {#1} \; \mid \; {#2} \;  \}}
\newcommand{\shortset}[2]{\{ {#1} \,|\, {#2}   \}}

\newcommand{\tensor}{\otimes}

\newcommand{\sprime}{\sp\prime}
\newcommand{\sptimes}{\sp{\times}}

\newcommand{\dual}{\sp{\vee}}
\newcommand{\inv}{\sp{-1}}

\newcommand{\bdr}{\partial}
\newcommand{\Hom}{\mathord{\mathrm {Hom}}}

\newcommand{\Ker}{\operatorname{\mathrm {Ker}}\nolimits}
\newcommand{\Coker}{\operatorname{\mathrm {Coker}}\nolimits}
\newcommand{\Image}{\operatorname{\mathrm {Im}}\nolimits}

\newcommand{\Spec}{\operatorname{\mathrm {Spec}}\nolimits}
\newcommand{\rank}{\operatorname{\mathrm {rank}}\nolimits}

\newcommand{\rmand}{\textrm{and}}

\newcommand{\quand}{\quad\rmand\quad}

\newcommand{\LLLKKK}{\LLL_{\KKK}}

\newcommand{\phim}{\phi}
\newcommand{\rhom}{\rho}
\newcommand{\rhomJ}{\rho_{J}}
\newcommand{\barR}{\widebar{R}}

\newcommand{\Pz}{\P}
\newcommand{\Pw}{\mathbf{P}}
\newcommand{\Cz}{\C}
\newcommand{\Cw}{\mathbf{C}}
\newcommand{\Ru}{\mathbf{R}}

\newcommand{\ii}{\sqrt{-1}}
\newcommand{\intp}[1]{\langle#1\rangle}

\newcommand{\sgn}{\mathord{\rm{sgn}}}

\newcommand{\affPi}{\Pi\sp{\circ}}

\newcommand{\ve}{\varepsilon}

\newcommand{\intpempty}{\intp{\phantom{\cdot},\phantom{\cdot}}}

\newcommand{\Tors}{\operatorname{Tors}}

\newcommand{\latin}[1]{\emph{#1}}
\newcommand{\ie}{\latin{i.e.}}
\newcommand{\widebar}[1]{\,\overline{\!#1}}

\renewcommand{\index}{\overline{n+1}}

\newcommand{\GLp}{\Lambda}
\def\ideal(#1|#2){(#1\,|\,#2)}

\newcommand{\module}{\mathcal{C}}
\newcommand{\bmodule}{\widebar{\module}}

\begin{document}

\title[Projective subspaces in
complex Fermat varieties]
{On the topology of
projective subspaces in  \\
complex Fermat varieties}

\author{Alex Degtyarev}

\address{%
Bilkent University\\
Department of Mathematics\\
06800 Ankara, TURKEY}
\email{
degt@fen.bilkent.edu.tr}

\author{Ichiro Shimada}

\address{
Department of Mathematics,
Graduate School of Science,
Hiroshima University,
1-3-1 Kagamiyama,
Higashi-Hiroshima,
739-8526 JAPAN
}
\email{
shimada@math.sci.hiroshima-u.ac.jp
}

\thanks{Partially supported by
 JSPS Grants-in-Aid for Scientific Research (C) No. 25400042,
 and
 JSPS Grants-in-Aid for Scientific Research (S) No. 22224001.
}

\begin{abstract}
Let  $X$ be the complex Fermat variety of dimension $n=2d$ and degree $m>2$.
We  investigate   the submodule of the middle homology group of $X$  with integer coefficients
generated by the classes of standard $d$-dimensional
subspaces contained in $X$,
and give
an algebraic (or rather combinatorial)
criterion for the primitivity of this submodule.
\end{abstract}

%14C30  	Transcendental methods, Hodge theory [See also 14D07, 32G20, 32J25, 32S35], Hodge conjecture
%14C17  	Intersection theory, characteristic classes, intersection multiplicities
% 	14Jxx		Surfaces and higher-dimensional varieties {For analytic theory, see 32Jxx}
%14J70  	Hypersurfaces

%Primary Classification
%14   (1940-now) Algebraic geometry
%14F   (1973-now) (Co)homology theory [See also 13Dxx]
%14F25   (1973-now) Classical real and complex (co)homology
%Secondary Classification
%14   (1940-now) Algebraic geometry
%14J   (1973-now) Surfaces and higher-dimensional varieties [For analytic theory, see 32Jxx]
%14J99   (1973-now) None of the above, but in this section
%32   (1940-now) Several complex variables and analytic spaces [For infinite-dimensional holomorphy, see 46G20, 58B12]
%32J   (1973-now) Compact analytic spaces [For Riemann surfaces, see 14Hxx, 30Fxx; for algebraic theory, see 14Jxx]
%32J99   (1973-now) None of the above, but in this section


\maketitle

\section{Introduction}
Unless specified otherwise, all (co-)homology groups are with coefficients
in~$\Z$.

Let $X$ be the complex Fermat variety
$$
z_0^m+\dots+z_{n+1}^m=0
$$
of dimension $n$ and degree $m>2$ in a projective space $\Pz^{n+1}$
with homogeneous coordinates $(z_0:\dots:z_{n+1})$.
Suppose that $n=2d$ is even.
Let $\JJJ$ be the set of
all unordered partitions of the index set $\index:=\{0,1,\ldots,n+1\}$
into unordered pairs, \ie,
lists
$$
J:=[[j_0, k_0], \dots, [j_d, k_d]]
$$
of pairs of indices such that
\begin{equation}\label{eq:condJ}
\{j_0, k_0, \dots, j_d, k_d\}=\index,\quad
j_i<k_i\;(i=0, \dots, d),
\quad
j_0<\dots<j_d,
\end{equation}
and let $\BBB$ be the set of $(d+1)$-tuples
$\beta=(\beta_0, \dots, \beta_d)$
of complex numbers $\beta_i$ such that $\beta_i^m=-1$.
(Note that we always have $j_0=0$.)
For $J\in \JJJ$ and $\beta\in \BBB$,
we define the \emph{standard $d$-space $L_{J, \beta}$}
to be  the
projective
subspace of $\Pz^{n+1}$ defined by the equations
\begin{equation}\label{eq:Lhomog}
z_{k_i}=\beta_i z_{j_i}\quad (i=0, \dots, d).
\end{equation}
The number of these spaces equals $(2d+1)!!\,m^{d+1}$,
where $(2d+1)!!$ is the product of all
\emph{odd} numbers from~$1$ to $(2d+1)$.
Each standard $d$-space
$L_{J, \beta}$ is contained in $X$,
and hence we have its class $[L_{J, \beta}]$ in the middle homology group $H_n(X)$ of $X$.
Let $\LLL(X)$ denote the $\Z$-submodule of $H_n(X)$ generated by the classes $[L_{J, \beta}]$
of all standard $d$-spaces.
\par
In the case $n=2$,
the problem to determine whether   $\LLL(X)$ is primitive in $H_n(X)$  or not
was raised by Aoki and Shioda~\cite{MR717587}
in the study of the Picard groups of Fermat surfaces.
In degrees~$m$ prime to~$6$,
the primitivity of $\LLL(X)$ implies that the Picard group of~$X$ is
generated by the classes of the lines contained in~$X$.
Sch\"{u}tt, Shioda and van Luijk~\cite{MR2653207}
studied this problem  using the  reduction of $X$ at supersingular primes.
Recently,
the first author of the present article
solved in~\cite{arXiv13053073} this problem affirmatively by means of  the Galois covering  $X\to \P^2$
and  the method of Alexander modules.
\par
The purpose of this paper is to study
the subgroup $\LLL(X)\subset H_n(X)$
for higher-dimensional Fermat varieties.
For a non-empty subset $\KKK$ of $\JJJ$,
we denote by
$\LLLKKK(X)$
the
$\Z$-submodule
of $H_n(X)$ generated by the classes $[L_{J, \beta}]$,
where $J\in \KKK$ and $\beta\in \BBB$.
\par
To state our results,
we prepare several polynomials in
$\Z[t_{1}, \dots, t_{n+1}]$,
rings, and modules.
We put
 $$
 \phim(t):=t^{m-1}+\cdots + t +1,
 \quad
 \rhom(x, y):=\sum_{\mu=0}^{m-2} x^{\mu} \left(\sum_{\nu=0}^{\mu}y^{\nu}\right).
 $$
For $J=[[j_0, k_0], \dots, [j_d, k_d]] \in \JJJ$,
 we put
\begin{eqnarray*}
 \tau_J &:=& (t_{k_0}-1)\cdots  ( t_{k_d}-1), \\
 \psi_J &:=& \tau_J\cdot\phi(t_{j_1}t_{k_1})\cdots\phi(t_{j_d}t_{k_d}), \\
 \rhomJ &:=& \rhom(t_{j_1}, t_{k_1}) \cdots \rhom(t_{j_d}, t_{k_d}).
\end{eqnarray*}
Consider the ring
$$
\GLp:=\Z[t_0^{\pm1},\ldots,t_{n+1}^{\pm1}]/(t_0\ldots t_{n+1}-1)=\Z[t_1^{\pm1},\ldots,t_{n+1}^{\pm1}]
$$
 of Laurent
polynomials
and let
\begin{eqnarray*}
R&:=&\GLp/(t_0^m-1, \dots, t_{n+1}^m -1)=\Z[t_{1}, \dots, t_{n+1}]/(t_1^m-1, \dots, t_{n+1}^m -1), \\
\barR&:=&R/( \phim(t_0), \dots, \phim(t_{n+1}))=\Z[t_{1}, \dots, t_{n+1}]/( \phim(t_1), \dots, \phim(t_{n+1})).
\end{eqnarray*}
For $J=[[j_0, k_0], \dots, [j_d, k_d]] \in \JJJ$,
 we put
\begin{eqnarray*}
R_J&:=&R/ (t_{j_1}t_{k_1}-1, \dots, t_{j_d}t_{k_d}-1), \\
\barR_J&:=&\barR/ (t_{j_1}t_{k_1}-1, \dots, t_{j_d}t_{k_d}-1).
\end{eqnarray*}
Note that we always have $t_{j_0}t_{k_0}-1=0$ in $R_J$ and $\barR_J$.
The
multiplicative identities of these rings, \ie, the images of
$1\in\GLp$ under the quotient projection, are denoted by $1_J$.
\par
Our primary concern is the structure of the \emph{abelian group}
$H_n(X)/ \LLLKKK(X)$. 
For this reason, whenever speaking about the
\emph{torsion} of an abelian group~$A$, we always mean its $\Z$-torsion
$\Tors A:=\Tors_{\Z}A$, even
if $A$ happens to be an $R$- or $\barR$-module.
(Over $R$, almost all our modules have torsion.)
Respectively, $A$ is said to be \emph{torsion free} if its $\Z$-torsion
$\Tors_{\Z}A$ is trivial.
\par
Our main results are as follows.
%With a certain abuse of the language, by
%the \emph{torsion} of an $R$- or $\barR$-module
%we mean the torsion of the underlying $\Z$-module.
%
\begin{theorem}[see Section~\ref{sec:LLL}]\label{thm:main}
Let $\KKK$ be a non-empty subset of $\JJJ$.
Then the torsion of  the quotient module $H_n(X)/ \LLLKKK(X)$ is isomorphic
to the torsions of
any of
the following modules:
\begin{itemize}
\item[(a)]
the ring
$R/\ideal(\psi_J | J \in \KKK)$,
where $\ideal(\psi_J | J \in \KKK)$
is the ideal of $R$ generated by $\psi_J$
with $J$ running  through $\KKK$,
\item[(b)]
the ring
$\barR/\ideal(\rhomJ | J \in \KKK)$, where $\ideal(\rhomJ | J \in \KKK)$
is the ideal of $\barR$ generated by $\rhomJ$
with $J$ running  through $\KKK$,
\item[(c)]
the $R$-module
$$
\module_{\KKK}:=\left(\bigoplus_{J\in \KKK} R_J\right) / \MMM,
$$
where $\MMM$ is the  $R$-submodule of $\bigoplus_{J\in \KKK} R_J$
generated by $\sum _{J\in \KKK} \tau_J 1_J $,
\item[(d)]
the $\barR$-module
$$
\bmodule_{\KKK}:=\left(\bigoplus_{J\in \KKK} \barR_J\right) / \widebar{\MMM},
$$
where $\widebar{\MMM}$ is the  $\barR$-submodule of $\bigoplus_{J\in \KKK} \barR_J$
generated by $\sum _{J\in \KKK}1_J $.
\end{itemize}
\end{theorem}
In particular,
we assert that the torsion parts of all four modules
listed in Theorem~\ref{thm:main}
are isomorphic, although
not always canonically: sometimes, we use the abstract
isomorphism $A\cong\Hom_{\Z}(A,\Q/\Z)$ for a finite abelian group~$A$,
see Section~\ref{subsec:proofcd} for details.
%Note that,
It is worth mentioning that,
according
to~\cite{arXiv13070382},
in the case $d=2$ of Fermat surfaces, the \latin{a priori} more complicated
%and found using
module dealt
with in~\cite{arXiv13053073} (which was found by means of a completely different
approach)
is isomorphic to the one that is
given in  Theorem~\ref{thm:main}(c).
\begin{conjecture}\label{conj:main}
If $\KKK=\JJJ$, the group $H_n(X)/ \LLLKKK(X)$ is torsion free.
\end{conjecture}
This conjecture is supported by some numerical evidence (see
Section~\ref{sec:comp} for details) and by the fact that it holds in the
cases
$d=0$ (obvious) and $d=1$ (see~\cite{arXiv13053073}).
%Note that, according
%to~\cite{arXiv13070382}, the \latin{a priori} more complicated module dealt
%with in~\cite{arXiv13053073}
%has torsion isomorphic to that of the module given by Theorem~\ref{thm:main}(c).
%
%\par
%
Theorem~\ref{thm:main}
reduces Conjecture~\ref{conj:main} to a purely
algebraic (or even combinatorial) question. However, for the moment it
remains open.
\begin{definition}\label{def:gamma}
Let $\mu_m$ be the subgroup $\shortset{z\in\C}{z^m=1}$ of $\C\sptimes$.
Denote by $\Gamma_{\KKK}$ the subset of
$\mu_m^{n+1}=\Spec (R\tensor\C)$
consisting of the elements $(a_1, \dots, a_{n+1})\in \mu_m^{n+1}$
such that $a_i\ne 1$ for $i=1, \dots, n+1$ and that
there exists $J=[[j_0, k_0], \dots, [j_d, k_d]]\in \KKK$ such that
$a_{j_{i}} a_{k_{i}}=1$ hold for $i=1, \dots,  d$.
\end{definition}
\begin{theorem}[see Section~\ref{subsec:proofrank}]\label{thm:mainrank}
For any  non-empty subset $\KKK$ of $\JJJ$,
the rank of
the group
$\LLLKKK(X)$ is equal to $|\Gamma_{\KKK}|+1$.
\end{theorem}
As a corollary, we obtain the following
statement,
which is a higher-dimensional generalization of Corollary 4.4 of~\cite{MR2653207}:
\begin{corollary}[see Section~\ref{subsec:proofrank}]\label{cor:primem}
For any  non-empty subset $\KKK$ of $\JJJ$,
the order of the torsion of $H_n(X)/\LLLKKK(X)$
may be divisible only by those primes that divide~$m$.
\end{corollary}
Applying Theorem~\ref{thm:main} to a subset $\KKK$ consisting of a single element
and using a deformation from  $X$,
we also prove the following generalization of Theorem 1.4 of~\cite{arXiv13053073}.
Let $f_i(x, y)$ be a homogeneous binary form of degree $m$
for $i=0, \dots, d$.
Suppose that the hypersurface $W$ in $\P^{n+1}$ defined by
\begin{equation}\label{eq:eqW}
f_0 (z_0, z_1) + f_1(z_2, z_3)+\cdots+ f_d (z_n, z_{n+1})=0
\end{equation}
is smooth. Then each $f_i(x, y)=0$ has $m$ distinct zeros
$(\alpha^{(i)}_{1}: \beta^{(i)}_{1}), \dots, (\alpha^{(i)}_{m}: \beta^{(i)}_{m})$
on $\P^1$.
Consider the points
$$
P^{(i)}_{\nu}:=(0: \dots: \underset{(2 i)}{\alpha^{(i)}_{\nu}}:   \underset{(2 i+1)}{\beta^{(i)}_{\nu}}:\dots :0)
$$
of $\P^{n+1}$.
Then, for each $(d+1)$-tuple $(\nu_0, \dots, \nu_{d})$ of integers $\nu_i$ with $1\le \nu_i  \le m$,
the $d$-space $L\sprime_{(\nu_0, \dots, \nu_{d})}$
spanned by $P^{(0)}_{\nu_0},  \dots, P^{(d)}_{\nu_d}$ is contained in $W$.
\begin{corollary}[see Section~\ref{subsec:proofcorW}]\label{cor:W}
The submodule of $H_n(W)$ generated by the classes
$[L\sprime_{(\nu_0, \dots, \nu_{d})}]$
of the $m^{d+1}$
subspaces $L\sprime_{(\nu_0, \dots, \nu_{d})}$ contained in $W$ is
of rank $(m-1)^{d+1}+1$ and is primitive
in $H_n(W)$.
\end{corollary}
The
last statement can further be extended to what we call a
\emph{partial Fermat variety},
\ie, a hypersurface $W_{s}\subset\P^{n+1}$
given by equation~\eqref{eq:eqW} with
$$
f_0(x,y)=\cdots=f_{s}(x,y)=x^m+y^m
$$
and the remaining polynomials
distinct (pairwise and from $x^m+y^m$) and generic.
Such a variety contains $(2s+1)!!\,m^{d+1}$ projective linear subspaces~$L'_*$ of dimension~$d$:
each subspace can be obtained as the projective span of one of the
$s$-spaces in the Fermat variety
$$
X(2s):=W_{s}\cap\{z_{2s+2}=\cdots=z_{n+1}=0\}\subset\P^{2s+1}
$$
and one of the $(d-s)$-tuples of points
$P^{(s+1)}_{\nu_{s+1}},  \dots, P^{(d)}_{\nu_d}$ as above.
Then, we have the following conditional statement.
\begin{corollary}[see Section~\ref{subsec:proofcorW}]\label{cor:Wext}
Assume that the statement of Conjecture~\ref{conj:main} holds for Fermat
varieties of dimension~$2s\ge0$.
Then, for any $d\ge s$,
the submodule of $H_n(W_{s})$ generated by the classes
$[L'_*]$
of the
subspaces $L'_*$ contained in $W_{s}$ is primitive
in $H_n(W_{s})$. In particular, this submodule is primitive for $s=0$
or~$1$.
\end{corollary}
\par
We conclude this introductory section with a very brief outline of the other
developments related to the subject.
\par
In~\cite{MR552586} and~\cite{MR526513},
the $\Q$-Hodge structure
on the rational cohomology
$H^n(X, \Q)$
was intensively investigated.
Letting $\zeta:=e^{2\pi\ii/m}$,
the tensor product $H^n(X)\tensor \Q(\zeta)$ decomposes into
simple representations of a certain abelian group $G$
(see Section~\ref{sec:outline} below), which are all of dimension~$1$ and
pairwise distinct. This decomposition is compatible with the Hodge
filtration, and the Hodge indices of the summands are computed explicitly.
As a by-product of this computation, one concludes that, at least if the
degree~$m$ is a prime,
the space of rational Hodge classes $H^{d,d}(X)\cap H^n(X, \Q)$
is generated by the
classes of the standard $d$-spaces. (See also Ran~\cite{MR594486}.)
(In the special case $d=1$ of surfaces,
this rational generation property holds for all degrees prime
to~$6$.)
It is this fact that motivates our work and makes the study of the torsion of
the quotient $H_n(X)/ \LLL_{\JJJ}(X)$ particularly important: if this torsion is
trivial, the classes of the standard $d$-spaces generate
the $\Z$-module  of \emph{integral} Hodge classes $H^{d,d}(X)\cap H^n(X, \Z)$.
\par
In~\cite{MR1794260},
we investigated the Fermat variety $X_{q+1}$
of even dimension  and degree $q+1$
in characteristic $p>0$,
where $q$ is a power of $p$.
By considering
the
middle-dimensional
subspaces contained in $X_{q+1}$,
we showed that the discriminant of the lattice
of numerical equivalence classes of middle-dimensional algebraic cycles
of  $X_{q+1}$ is a power of $p$.
Note that the rank of this lattice is equal to the middle Betti
number of $X_{q+1}$,
that is, $X_{q+1}$ is supersingular.
\par
In~\cite{MR2789841},
we
suggested
a general method to calculate the primitive closure
in $H^2(Y)$ of the lattice generated by the classes of given curves
on a complex algebraic surface $Y$.
As an example, we applied
this method
to certain branched covers of the complex projective plane.
\par
In~\cite{arXiv13070382},
the method of~\cite{arXiv13053073} was generalized to the calculation of
the Picard groups
of the so-called \emph{Delsarte surfaces} $Y$.
More precisely, the computation of the Picard rank was suggested
in~\cite{MR833362}, and \cite{arXiv13070382} deals with
the (im-)primitivity of the subgroup $\LLL(Y)\subset H_2(Y)$ generated by the
classes of certain ``obvious'' divisors.
In a few cases, this subgroup is primitive, but as a rule the quotient
$H_2(Y)/\LLL(Y)$ does have a certain controlled torsion.
\par
\medskip
{\noindent\bf Acknowledgements.}
The authors heartily thank Professor Tetsuji Shioda
for many discussions.
This
work was partially completed during the first author's visit to
Hiroshima University; we extend our gratitude to this institution for its
great hospitality.
\par
\medskip
{\noindent\bf Notation.}
By $(a, \dots, \underset{(i)}{b}, \dots a)$,
we denote a vector whose $i$th coordinate is $b$ and other coordinates are $a$.
The hat $\hat{\phantom{a}}$ means
omission of an element; for example,
by $(a_1, \dots, \hat{a_i}, \dots, a_N)$,
we denote the vector $(a_1, \dots, a_{i-1}, a_{i+1} \dots, a_N)$.
\section{An outline of the proof}\label{sec:outline}
To avoid confusion, let us denote by $\Pw^{n+1}$
another copy of the projective space, the one
with homogeneous coordinates $(w_0:\dots:w_{n+1})$.
(Below, we will also use $\Cw^{n+1}$ for
an affine chart of~$\Pw^{n+1}$.)
In $\Pw^{n+1}$,
consider the hyperplane $\Pi$ defined by
$$
w_0+\dots+ w_{n+1}=0.
$$
Then we have the  Galois covering $\pi\colon X\to \Pi$
defined by
$$
(z_0:\dots:z_{n+1})\mapsto (z_0^m:\dots:z_{n+1}^m).
$$
We put
$\zeta:=e^{2\pi\ii/m}$.
Then the Galois group $G$ of $\pi$ is generated by
$$
\gamma_i\colon  (z_0:\dots: z_i: \dots :z_{n+1})\mapsto (z_0:\dots:\zeta z_i: \dots:z_{n+1})
$$
for $i=0, \dots, n+1$.
Since  $\gamma_0\cdots \gamma_{n+1}=1$,
this group $G$ is isomorphic to $(\Z/m\Z)^{n+1}$.
Throughout this paper,
we regard $R$ as the group ring $\Z[G]$
by corresponding $\gamma_i\in G$ to the variable $t_i$
for $i=1, \dots, n+1$, and $\gamma_0\in G$ to $t_0=t_1\inv \cdots t_{n+1}\inv$.
Then we can regard $H_n(X)$ as an $R$-module.
Note that,
for any subset $\KKK$ of $\JJJ$,
the
subgroup
$\LLLKKK(X)$ of $H_n(X)$ is in fact an  $R$-submodule,
because, for any $J\in \JJJ$,  $g\in G$, and $\beta\in \BBB$,
there exists $\beta\sprime\in \BBB$ such that $g(L_{J, \beta})=L_{J, \beta\sprime}$.
\par
Let $Y_0$ be the hyperplane section of $X$ defined by $ \{z_0=0\}$,
which is $G$-invariant.
Since the fundamental classes $[X]\in H_{2n}(X)$
and $[Y_0]\in H_{2n-2}(Y_0)$ are also fixed by~$G$, the
Poincar\'{e}--Lefschetz duality isomorphisms
$$
H_n(X\setminus Y_0)= H^n(X, Y_0),
\quad
H_{2n-i}(X)= H^{i}(X),\quad
H_{2n-2-i}(Y_0)= H^{i}(Y_0)
$$
are $R$-linear; hence, they convert the cohomology exact sequence of the pair
$(X,Y)$ into a long exact sequence \emph{of $R$-modules}
\begin{equation} \label{eq:longexact}
\cdots
 \maprightsp{}
 H_{n-1}(Y_0) \maprightsp{\bdr}
 H_n(X\setminus Y_0)   \maprightsp{\iota_*}
 H_n(X)  \maprightsp{}
 H_{n-2}(Y_0)\maprightsp{}
\cdots,
\end{equation}
where $\iota\colon X\setminus Y_0\inj X$ is the inclusion.
We then put
$$
V_n(X):=\Image (\iota_*\colon   H_n(X\setminus Y_0)\to  H_n(X) ).
$$
Since
the group
$H_{n-2}(Y_0)$ is torsion free,
the $R$-submodule $V_n(X)$ of $H_n(X)$ is primitive in $H_n(X)$ as a $\Z$-submodule.
\par
The structure of the $R$-module  $V_n(X)$  is
given
by the theory of
\emph{Pham polyhedron} developed in~\cite{MR0195868}.
Let
$z_0=1$ and regard $(z_1, \dots, z_{n+1})$ as affine coordinates on
the affine space $\Cz^{n+1}:=\Pz^{n+1}\setminus\{z_0=0\}$,
in which $X\setminus Y_0$ is defined by
$$
1+z_1^m+\cdots +z_{n+1}^m=0.
$$
Fix the $m$-th root $\eta:=e^{\pi\ii/m}$ of $-1$,
and consider the (topological) $n$-simplex
$$
D:=\set{(s_1\eta, \dots, s_{n+1}\eta)}
 {s_i\in\R,\;\;s_1^m+\dots+s_{n+1}^m=1, \;\; 0\le s_i\le 1}
$$
in $X\setminus Y_0$,
oriented so that
that,
if we consider $(s_1, \dots, s_{n})$ as local real coordinates of $D$ at an interior point  of $D$, then
\begin{equation*}\label{eq:bdrti}
(-\partial/\bdr s_1, \dots, -\partial/\bdr s_{n})
\end{equation*}
is a positively-oriented basis of the real tangent space of $D$ at this point.
Then it is easy to see that
the chain
$$
S:=(1-\gamma_1\inv)\dots (1-\gamma_{n+1}\inv) D
$$
is a cycle; moreover, it is homeomorphic to
the join of $(n+1)$ copies of the two-point space $\{\eta,\zeta\eta\}$, i.e.,
to the $n$-sphere.
(Here and below, we do not distinguish between ``simple'' singular chains
in~$X$ and the corresponding geometric objects, \latin{viz.} unions of
simplices with the orientation taken into account and the common parts of the
boundary identified. For this reason, we freely apply the module notation to
simplices.)
Hence,
we have the class $[S]\in H_n(X \setminus Y_0)$
and its image $[S]\in V_n(X)$ by $\iota_*$.
Pham~\cite{MR0195868} proved the following:
\begin{theorem}[see \cite{MR0195868}]\label{thm:pham}
The  homomorphism $1\mapsto [S]$ from $R$ to $H_n(X\setminus Y_0)$
induces an isomorphism $\barR  \cong H_n(X\setminus Y_0)$ of $R$-modules,
and hence a surjective  homomorphism $R\surj V_n(X)$ of $R$-modules.
\end{theorem}
The Poincar\'{e} duality
gives rise to symmetric bilinear pairings $\intpempty$
on the groups $H_n(X\setminus Y_0)$, $V_n(X)$, and $H_n(X)$, which is
interpreted geometrically as the signed intersection of $n$-cycles brought to a
general position.
We emphasize that these pairings are $\Z$-bilinear and
$G$-invariant (as so is the fundamental class $[X]$).
The homomorphisms  $H_n(X\setminus Y_0)\surj V_n(X)\inj H_n(X)$ preserve $\intpempty$.
Note  that $\intpempty$
is non-degenerate on $H_n(X)$, but not on $H_n(X\setminus Y_0)$.
Later, we will see that $\intpempty$ is also nondegenerate on $V_n(X)$.
\par
The main ingredient of the proof of Theorems~\ref{thm:main} and~\ref{thm:mainrank}  is the following:
\begin{theorem}[see Section~\ref{sec:LS}]\label{thm:LS}
For  $\beta_i\in \C\sptimes$ with $\beta_i^m=-1$,
we put
$$
s(\beta_i):=\begin{cases}
1 & \textrm{if $\beta_i=\eta$}, \\
-1 & \textrm{if $\beta_i=\eta^{-1}$}, \\
0 & \textrm{otherwise}.
\end{cases}
$$
(Recall that we fixed $\eta:=e^{\pi\ii/m}$.)
For $J=[[j_0, k_0], \dots, [j_d, k_d]]\in \JJJ$
ordered as in~\eqref{eq:condJ},
let $\sigma_J$ be the permutation
$$
\left(
\begin{array}{cccccc}
0 & 1 & \dots & \dots & n &n+1 \\
j_0 & k_0 & \dots & \dots & j_d & k_d
\end{array}
\right).
$$
Then we have
$$
\intp{L_{J, \beta}, S}=\sgn(\sigma_J) s(\beta_0) \cdots s(\beta_d),
$$
where $\beta=(\beta_0, \dots, \beta_d)\in \BBB$.
\end{theorem}
We use
Theorem~\ref{thm:LS} and the fact that the pairing on $H_n(X)$ is
nondegenerate to compute the subgroup $\LLLKKK (X)\subset H_n(X)$.
Various stages of this computation result in most principal statements of the
paper.
\section{Intersection of $S$ and the standard $d$-spaces}\label{sec:LS}
In this section, we prove Theorem~\ref{thm:LS}.
The affine part $X\setminus Y_0$ of $X$
is defined by $1+z_1^m+ \cdots+z_{n+1}^m=0$
in the affine space $\Cz^{n+1}$ with coordinates $(z_1, \dots, z_{n+1})$.
We put
$$
\Cw^{n+1}:=\Pw^{n+1}\setminus \{w_0=0\},
$$
and setting $w_0=1$,
we regard $(w_1, \dots, w_{n+1})$ as affine coordinates of $\Cw^{n+1}$.
We put
$$
z_i=x_i+\ii y_i,
\quad
w_i=u_i+\ii v_i,
$$
where $x_i, y_i, u_i, v_i$ are real coordinates.
Consider  the affine hyperplane
$$
\affPi:=\Pi\cap \Cw^{n+1}=\{ 1+ w_1+\cdots+ w_{n+1}=0\}
$$
of $\Cw^{n+1}$.
In the real part
$$
\affPi\cap \{v_1=\dots=v_{n+1}=0\}= \set{(u_1, \dots, u_{n+1} )\in \Ru^{n+1}}{1+u_1+\dots+u_{n+1}=0}
$$
of $\affPi$,
we have an $n$-simplex $\Delta$ defined by
$$
1+ u_1+\cdots+ u_{n+1}=0\quand
-1\le u_i\le 0\;\;\textrm{for}\;\; i=1, \dots, n+1.
$$
Then $\pi\colon X\to \Pi$ induces a homeomorphism $\pi|_D\colon D\isom \Delta$.
We put
$$
p_i:=(0, \dots, \underset{(i)}{\eta}, \dots 0)\in D,
$$
and put $\bar{p}_i:=\pi(p_i)=(0, \dots,\underset{(i)}{ -1}, \dots, 0)$.
Then $\bar{p}_1, \dots, \bar{p}_{n+1}$ are the vertices of $\Delta$.
\begin{remark}\label{rem:Sintpiinv}
Note that $S\subset \pi\inv(\Delta)$,
and that
\begin{equation*}\label{eq:Sintpiinv}
S\cap \pi\inv(\{\bar{p}_1, \dots, \bar{p}_{n+1}\})
=\{p_1, \gamma_1\inv(p_1), \dots, p_{n+1}, \gamma_{n+1}\inv(p_{n+1})\}.
\end{equation*}
\end{remark}
\begin{remark}\label{rem:orientation}
By the definition of the orientation of $D$ given in
Section~\ref{sec:outline},
we see that, locally at $p_i$,
the $n$-chain $D$ is identified with the product
$$
(-1)^{i+1}\; \overrightarrow{p_i p_1} \times \cdots \times
\overrightarrow{p_i p_{i-1}} \times \overrightarrow{p_i p_{i+1}}
\times\cdots\times \overrightarrow{p_i p_{n+1}}
$$
of $1$-chains,
where $\overrightarrow{p_i p_k}$ is the $1$-dimensional edge of $D$
connecting $p_i$ and $p_k$ and oriented from $p_i$ to $p_k$.
\end{remark}
\par
By the condition~\eqref{eq:condJ} on $\JJJ$,
we always have $j_0=0$.
Let $b_0$ be an element of $\Z/m\Z$ such that
$$
\beta_0=\eta^{1+2b_0}=\zeta^{b_0}\eta.
$$
In the affine coordinates $(z_1, \dots, z_{n+1})$ of $\Cz^{n+1}$,
the equations~\eqref{eq:Lhomog} of $L_{J, \beta}$ are written as
\begin{equation}\label{eq:Linhomog}
z_{k_0}=\beta_0, \quad z_{k_i}=\beta_i z_{j_i} \quad(i=1, \dots, d).
\end{equation}
If~\eqref{eq:Linhomog} holds,
then we have $z_{k_i}^m=- z_{j_i}^m$ for $i=1, \dots, d$,
and hence $L_{J, \beta} \cap \pi\inv (\Delta)$ consists of a single point
$$
(0, \dots, \underset{(k_0)}{\beta_0}, \dots 0)=\gamma_{k_0}^{b_0} (p_{k_0})
$$
by Remark~\ref{rem:Sintpiinv}.
Therefore, we have
$$
L_{J, \beta} \cap S=
\begin{cases}
\emptyset & \textrm{if $\beta_0\ne \eta$ and $\beta_0\ne \eta\inv$},\\
\{p_{k_0}\}& \textrm{if $\beta_0= \eta$},\\
\{\gamma_{k_0}\inv  (p_{k_0})\} & \textrm{if $\beta_0= \eta\inv$}.
\end{cases}
$$
In particular,
we have
\begin{equation}\label{eq:LSa}
\intp{L_{J, \beta}, S}=0\quad\textrm{if $\beta_0\ne \eta$ and $\beta_0\ne \eta\inv$}.
\end{equation}
\par
In order to calculate $\intp{L_{J, \beta}, S}$ in the cases
where $\beta_0=\eta\sp{\pm 1}$, we need the following lemma.
For an angle $\theta$,
we consider the oriented real semi-line
$$
H(\theta):=\R_{\ge 0}\,e^{\ii\theta}\quad
\textrm{with the orientation from $0$ to $e^{\ii\theta}$}
$$
on the complex plane $\C$,
and define the chain (with closed support)
$$
W(\theta):=H(\theta)-H(\theta-2\pi/m) =(1-\gamma\inv) H(\theta),
$$
where $\gamma\colon \Cz\to \Cz$ is the multiplication by  $\zeta=e^{2\pi\ii/m}$.
Note that $W(\pi/m)=H(\pi/m)-H(-\pi/m)$.
Let $\C^2$ be equipped with coordinates $(z, z\sprime)$.
For $\beta_i\in \C$ with $\beta_i^m=-1$,
we denote by $\Lambda_{\beta_i}$ the linear subspace of $\C^2$
defined by $z\sprime=\beta_i z$.
\begin{lemma}\label{lem:LambdaW}
The local intersection number $\ell(\beta_i)$ at the origin
in $\C^2$
of
the chains
$W(\pi/m)\times W(\pi/m)$ and
$\Lambda_{\beta_i}$
is equal to $s(\beta_i)$.
\end{lemma}
\begin{proof}
The linear subspace  $\Lambda_{\beta_i}$ is the graph of the function $f\colon z\mapsto z\sprime=\beta_i z$,
and hence $f(W(\pi/m))$ is obtained by rotating $W(\pi/m)$ by $\beta_i\in \C\sptimes$.
Let $\ve$ and $\ve\sprime$ be sufficiently small positive real numbers.
We perturb $\Lambda_{\beta_i}$
locally at the origin to the graph $\tilde\Lambda_{\beta_i}$ of the function
$$
\tilde{f}\colon z\mapsto z\sprime=\beta_i z+\ve e^{\ii \tau}\rho(|z|),
$$
where $\rho\colon \R_{\ge 0}\to \R_{\ge 0}$ is the function
$$
\rho(x)=\begin{cases}
1 & \textrm{if $x\le \ve\sprime$,}\\
2-x/\ve\sprime & \textrm{if $\ve\sprime\le x\le 2\ve\sprime$,}\\
0 & \textrm{if $2\ve\sprime\le x$.}
\end{cases}
$$
The direction $\tau$ of the perturbation is given
as in Figure~\ref{fig:WLambda},
where $W(\pi/m)$ are drawn by thick arrows,
$f(W(\pi/m))$ are drawn by  thin arrows
 and $\tilde f(W(\pi/m))$ are drawn by  broken arrows.
\begin{figure}
\includegraphics[width=10cm]{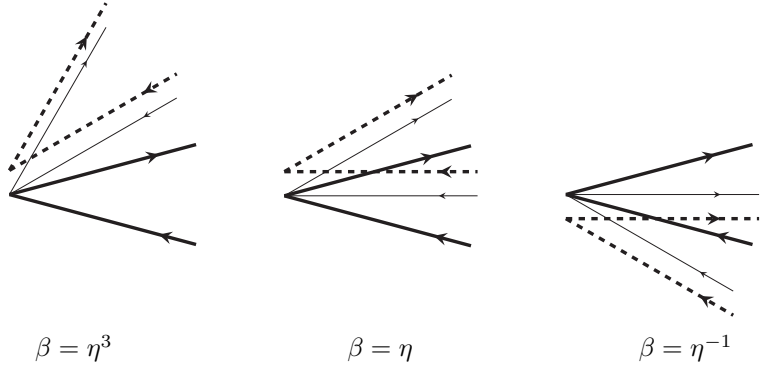}
\vskip .1cm
$\beta=\eta^3$
\hskip 3cm
$\beta=\eta$\hskip 3cm
$\beta=\eta\inv$
\caption{$W(\pi/m)$, $f(W(\pi/m))$ and  $\tilde f(W(\pi/m))$}\label{fig:WLambda}
\end{figure}
\par
Suppose that $\beta_i\ne \eta$ and $\beta_i\ne \eta\inv$.
As  Figure~\ref{fig:WLambda}  illustrates in the case $\beta_i=\eta^3$, we see that
$\tilde f(W(\pi/m))$ and $ W(\pi/m)$ are disjoint,
and hence
$$
\tilde\Lambda_{\beta_i}\cap (W(\pi/m)\times W(\pi/m))=\emptyset.
$$
Therefore $\ell(\beta_i)=0$.
\par
Suppose that $\beta_i= \eta$.
Then the intersection of $\tilde\Lambda_{\eta}$ and $W(\pi/m)\times W(\pi/m)$
consists of a single point
$(Q, \tilde f (Q))$,
where $Q\in H(-\pi/m)$ and $\tilde f (Q)\in H(\pi/m)$.
We choose a positively-oriented basis of the real tangent space of $\Cz^2$ at this point as
$$
(\partial/ \bdr x, \partial/ \bdr y, \partial/ \bdr x\sprime, \partial/ \bdr y\sprime),
\quad\textrm{where \;\;$z=x+\ii y, \;\;\;z\sprime=x\sprime+\ii y\sprime$}.
$$
The positively-oriented basis of the tangent space of $\tilde\Lambda_{\eta}$
at $(Q, \tilde f (Q))$ is
$$
(1, 0, \cos(\pi/m), \sin(\pi/m)),
\quad
(0, 1,  -\sin(\pi/m), \cos(\pi/m)),
$$
while the positively-oriented basis of the tangent space of $W(\pi/m)\times W(\pi/m)$
at $(Q, \tilde f (Q))\in H(-\pi/m)\times H(\pi/m)$ is
$$
(-\cos(-\pi/m), -\sin(-\pi/m), 0,0),
\quad
(0, 0,   \cos(\pi/m), \sin(\pi/m)).
$$
(Note that $W(\pi/m)$ is oriented \emph{toward} the origin on $H(-\pi/m)$.)
Calculating the sign of the determinant of the $4\times 4$ matrix with row vectors being the
four vectors above in this order,
we see that $\ell(\eta)=1$.
\par
Suppose that $\beta_i= \eta\inv$.
Then  $\tilde\Lambda_{\eta\inv} \cap W(\pi/m)\times W(\pi/m)$
consists of a single point
$(Q, \tilde f (Q))$,
where $Q\in H(\pi/m)$ and $\tilde f (Q)\in H(-\pi/m)$.
The positively-oriented basis of the tangent space of $\tilde\Lambda_{\eta\inv}$
at $(Q, \tilde f (Q))$ is
$$
(1, 0, \cos(-\pi/m), \sin(-\pi/m)),
\quad
(0, 1,  -\sin(-\pi/m), \cos(-\pi/m)),
$$
while that of $W(\pi/m)\times W(\pi/m)$
at $(Q, \tilde f (Q))\in H(\pi/m)\times H(-\pi/m)$ is
$$
(\cos(\pi/m), \sin(\pi/m), 0,0),
\quad
(0, 0,   -\cos(-\pi/m), -\sin(-\pi/m)).
$$
Calculating the determinant,
we see that $\ell(\eta\inv)=-1$.
\end{proof}
Let $p$ be $p_i$ or $\gamma_i\inv (p_i)$.
In  a small neighborhood $\UUU_p$ of $p$ in $X\setminus Y_0$,
we have   local coordinates
$(z_1, \dots, \hat{z}_i, \dots, z_{n+1})$
of $X\setminus Y_0$.
Let
$$
\iota_p\colon\;\;  \UUU_p\inj \Cz\times \cdots\times \Cz \quad \text{($n$ factors)}
$$
be the open immersion defined by $(z_1, \dots, \hat{z}_i, \dots, z_{n+1})$.
We consider an element
$$
g:=\gamma_1^{\nu_1}\cdots \gamma_{n+1}^{\nu_{n+1}}\;\in\; G,
$$
and give a local description of $g(D)$  at $p=p_i$ and  $p=\gamma_i\inv (p_i)$
via $\iota_p$.
\par
\medskip
(1)  Locally around $p=p_i$.
If $\nu_i\ne 0$,
then $p_i\notin g(D)$ and hence $\UUU_p\cap g(D)=\emptyset$.
Suppose that
$\nu_i=0$.
Using Remark~\ref{rem:orientation} and the fact that $g$ preserves the orientation,
we see that $g(D)$ is identified with
\begin{equation}\label{eq:pigD}
\renewcommand{\arraystretch}{1.4}
\begin{array}{c}
(-1)^{i+1}\;H((2\nu_1+1)\pi  /m)\times \cdots \times H((2\nu_{i-1}+1)\pi /m)\times\qquad\quad \\
\phantom{aaaaaaaaaaaaaa} H((2\nu_{i+1}+1)\pi ) /m)\times \cdots \times H((2\nu_{n+1}+1)\pi ) /m).
\end{array}
\end{equation}
\par
(2)  Locally around $p=\gamma\inv(p_i)$.
If $\nu_i\ne -1$,
then $\gamma\inv(p_i)\notin g(D)$ and hence $\UUU_p\cap g(D)$ is empty.
Suppose that
$\nu_i=-1$.
Then $g(D)$ is identified with~\eqref{eq:pigD}
because the action of $\gamma_i$ maps the local descriptions of $g(D)$
at $\gamma_i\inv (p_i)$ to that of
$\gamma_i g(D)$  at $p_i$.
\par
\medskip
We put
$$
S_i:=(1-\gamma_1\inv)\cdots (1-\gamma_{i-1}\inv) (1-\gamma_{i+1}\inv)\cdots (1-\gamma_{n+1}\inv) D
%,
$$
(note that $\gamma_i$ is missing),
which is a hemisphere of the $n$-sphere $S$ containing $p_i$.
The other hemisphere is $\gamma_i\inv (S_i)$,
and
we have
 $S=S_i-\gamma_i\inv (S_i)$.
Since $p_i \in S_i$ and $p_i\notin \gamma_i\inv(S_i)$,
$S$ is identified with
$$
(-1)^{i+1}\;W(\pi/m) \times \cdots\times W(\pi/m)
$$
locally at $p_i$ by $\iota_{p_i}$; while
since
$\gamma_i\inv (p_i) \notin S_i$ and $\gamma_i\inv (p_i)\in \gamma_i\inv(S_i)$,
$S$ is identified with
$$
-(-1)^{i+1}\;W(\pi/m) \times \cdots\times W(\pi/m)
$$
locally at $\gamma_i\inv(p_i)$ by $\iota_{\gamma_i\inv(p_i)}$.
\par
Suppose that $\beta_0=\eta$.
We calculate the local intersection number of $L_{J, \beta}$ and $S$ at $p:=p_{k_0}$.
As was shown above,
the topological $n$-cycle $S$ is identified locally at $p$
with
$$
(-1)^{k_0+1} W(\pi/m) \times \cdots\times W(\pi/m)
$$
by  the local coordinates $(z_1, \dots, \hat{z_{k_0}}, \dots, z_{n+1})$
of $X\setminus Y_0$ with the origin  $p$.
Note that $\{1, \dots, \hat{k_0}, \dots, n+1\}$ is equal to $\{j_1, k_1, \dots, j_d, k_d\}$.
We permute the  coordinate system  $(z_1, \dots, \hat{z_{k_0}}, \dots, z_{n+1})$
to
$$
(z_{j_1}, z_{k_1}, \dots, z_{j_d}, z_{k_d}),
$$
and define a new open immersion
$$
\iota_p\sprime \colon\;\; \UUU_p\;\inj\;
\overbrace{\C^{\vphantom{2}}\times \cdots\times \C}^{\text{$n$ times}} =\overbrace{\C^2\times \dots \times \C^2}^{\text{$d$ times}}
$$
by this new coordinate system.
By $\iota_p\sprime$,
the topological $n$-cycle $S$ is identified locally at $p$
with
$$
(-1)^{k_0+1}\,\sgn (\sigma_J\sprime)\,  W(\pi/m) \times \cdots\times W(\pi/m),
$$
where $\sigma_J\sprime$ is the permutation
$$
\left(
\begin{array}{ccccccc}
1 & &\dots& \hat{k_0}&\dots & n &n+1 \\
j_1 & k_1 &\dots & & \dots & j_d & k_d
\end{array}
\right).
$$
On the other hand,
$L_{J, \beta}$ is identified by $\iota_p\sprime $
with
$$
\Lambda_{\beta_1}\times \dots\times \Lambda_{\beta_d}
$$
locally at $p$.
By Lemma~\ref{lem:LambdaW},
we have
\begin{equation}\label{eq:LSb}
\intp{L_{J, \beta}, S}=(-1)^{k_0+1}\sgn (\sigma_J\sprime) s(\beta_1)\cdots s(\beta_d)
\quad\textrm{if $\beta_0=\eta$}.
\end{equation}
\par
Suppose that $\beta_0=\eta\inv$.
We calculate the local intersection number of $L_{J, \beta}$ and $S$ at $p:=\gamma_{k_0}\inv (p_{k_0})$.
As was shown above,
the new open immersion $\iota_p\sprime$
identifies $S$  with
$$
-(-1)^{k_0+1}\sgn (\sigma_J\sprime)  W(\pi/m) \times \cdots\times W(\pi/m),
$$
locally at $p$.
Calculating as above,
we have
\begin{equation}\label{eq:LSc}
\intp{L_{J, \beta}, S}=-(-1)^{k_0+1}\sgn (\sigma_J\sprime) s(\beta_1)\cdots s(\beta_d)
\quad\textrm{if $\beta_0=\eta\inv$}.
\end{equation}
\par
The dependence on $\beta_0$ in the right-hand sides of~\eqref{eq:LSa},~\eqref{eq:LSb},~\eqref{eq:LSc}
can be expressed by the
extra
factor $s(\beta_0)$.
Observing that
$(-1)^{k_0+1} \sgn (\sigma_J\sprime)=\sgn (\sigma_J)$,
we complete the proof of Theorem~\ref{thm:LS}.
\qed
\section{The $R$-submodule $\LLLKKK(X)$}\label{sec:LLL}
\subsection{Preliminaries}
For an $R$-module $M$, we put $M\dual:=\Hom_{\Z}(M, \Z)$,
which is regarded as an $R$-module
\latin{via}
the contragredient action of $G$ on $M\dual$.
\par
Let $M$ be a finitely generated $\Z$-module.
We put $d_M:=\rank M=\dim_{\Q} M\tensor \Q$.
Note that $M$ is torsion free if and only if it can be generated by $d_M$
elements.
%Recall that, when speaking about the torsion of an $R$-module~$M$, we always
%mean its \emph{integral} torsion
%$\Tors M:=\Tors_{\Z}M$.
%
\begin{lemma}\label{lem:A}
Let $x_1, \dots, x_N$ be variables.
We put
$$
A:=\Z[x_1, \dots, x_N]/(x_1^m-1, \dots, x_N^m-1),
$$
and $\theta:=(x_1-1)\cdots (x_N-1)$.
Then $A/(\theta)$ is torsion free as a $\Z$-module.
Moreover  the annihilator ideal of $\theta$ in $A$ is generated by
$\phim(x_1), \dots, \phim(x_N)$.
\end{lemma}
\begin{proof}
We fix the monomial order
{\tt grevlex} on $\Z[x_1, \dots, x_N]$ (see~\cite[Chapter 2]{MR1417938}).
Since the leading coefficients of $x_1^m-1, \dots, x_N^m-1$ and $\theta$ are $1$,
the division algorithm by the set of these polynomials can be carried out over $\Z$.
Then we see that
$A/(\theta)$ is generated as a $\Z$-module by
\begin{equation}\label{eq:genmons}
x_1^{\nu_1}\cdots x_{N}^{\nu_N}\quad \textrm{with $0\le \nu_i<m$ for all $i$ and $\nu_i=0$ for at least one $i$}.
\end{equation}
On the other hand, the reduced $0$-dimensional scheme
$\Spec (A/(\theta)\tensor \C)$
consists of the closed points
\begin{equation}\label{eq:genpts}
(a_1, \dots, a_N)\in \mu_m^N\quad \textrm{with  $a_i=1$ for at least one $i$}.
\end{equation}
The number of monomials in~\eqref{eq:genmons} is equal to the number of points in~\eqref{eq:genpts},
and the latter is equal to $d_{A/(\theta)}$.
Hence, by the observation above, we see that  $A/(\theta)$ is torsion free.
The second part also follows from the division algorithm over $\Z$ by $\{\phim(x_1), \dots, \phim(x_N)\}$ of monic polynomials
of degree $m-1$.
\erase{
\par
It is obvious that $\phim(x_i)\theta=0$ on $A$.
Suppose that $f(x_1, \dots, x_N)\in A$ is an annihilator of $\theta$.
Applying the division algorithm over $\Z$ by $\{\phim(x_1), \dots, \phim(x_N)\}$ of monic polynomials
of degree $m-1$,
we can assume that
$$
\deg_{x_i} f<m-1\quad \textrm{for $i=1, \dots, N$}.
$$
Suppose that $f\ne 0$.
then we have
$$
\deg_{x_i} f\theta<m\quad \textrm{for $i=1, \dots, N$},
$$
and hence $f\lambda\ne 0$ in $A$, which is absurd.
}
\end{proof}
\subsection{Proof of  Part (a) of Theorem~\ref{thm:main}}\label{subsec:proofa}
We define a non-degenerate symmetric bilinear form
$[\phantom{\cdot}, \phantom{\cdot}]\colon R\times R\to \Z$
by
$$
[\, t_1^{\nu_1}\cdots t_{n+1}^{\nu_{n+1}}\, , \, t_1^{\nu_1\sprime}\cdots t_{n+1}^{\nu_{n+1}\sprime}\,]:=
\delta_{\nu_1\nu_1\sprime} \dots \delta_{\nu_{n+1}\nu_{n+1}\sprime},
$$
where $\delta_{ij}$ is the Kronecker delta on $\Z/m\Z$.
Since $[\phantom{\cdot}, \phantom{\cdot}]$
obviously is unimodular and
satisfies $[gf, gf\sprime]=[f, f\sprime]$ for $f, f\sprime\in R$ and $g\in G$,
it induces an isomorphism $R \cong R\dual$ of $R$-modules.
Note that the image of the dual homomorphism $f\dual\colon M\dual \to R$ of
an $R$-linear  homomorphism
$f\colon R\to M$
is an ideal of $R$, and the cokernel of $f\dual$ is always torsion free,
because
$$
\Image f\dual
=\set{x\in R}{[x, y]=0\;\textrm{for any}\; y\in \Ker f}.
$$
\par
In particular,
the surjective homomorphism $R\surj V_n(X)$ in Theorem~\ref{thm:pham}
defines an ideal $ V_n(X)\dual\inj R$ of $R$  such that $R/V_n(X)\dual$ is torsion free
as a $\Z$-module.
On the other hand, the $G$-invariant  intersection pairing $\intpempty$ defines an isomorphism $H_n(X)\cong H_n(X)\dual$ of $R$-modules.
Hence we obtain the dual  homomorphism $H_n(X) \to V_n(X) \dual$ of $V_n(X) \inj H_n(X)$,
which is surjective because $V_n(X)$ is primitive in $H_n(X)$~(see~\eqref{eq:longexact}).
By construction, the composite $H_n(X) \to R$  of the two homomorphisms
$H_n(X) \surj V_n(X) \dual$ and $ V_n(X)\dual\inj R$
maps $\tau \in H_n(X) $ to
$$
\sum_{\nu_1, \dots, \nu_{n+1}\in \Z/m\Z} \;\; \intp{\,\tau, \gamma_1^{\nu_1} \cdots \gamma_{n+1}^{\nu_{n+1}} (S) \,}  \;\cdot\;
t_1^{\nu_1} \cdots t_{n+1}^{\nu_{n+1}} \;\; \in \;\; R.
$$
\par
Consider the composite
$$
\LLLKKK(X) \inj H_n(X) \surj V_n(X) \dual,
$$
where  the second homomorphism  is the dual of $V_n(X) \inj H_n(X)$.
Let $\LLLKKK\sprime(X)$ be the image of this composite.
We
have
the following:
\begin{claim}\label{claim:LLLKKK}
One has
$\rank \LLLKKK(X)=\rank \LLLKKK\sprime(X)+1$, and
$$
H_n(X)/\LLLKKK(X)\cong V_n(X) \dual/\LLLKKK\sprime(X).
$$
\end{claim}
\begin{proof}
Let
$P_X\in H_n(X)$ denote the class of the intersection of $X$
and a $(d+1)$-dimensional
subspace of $\Pz^{n+1}$.
By the Lefschetz hyperplane section theorem, the kernel of $H_n(X)  \surj  V_n(X) \dual$ is
$\Z P_X$.
Therefore it is enough to show that $\LLLKKK(X)$ contains $P_X$.
Since $\KKK$ is non-empty,
we can assume by  a permutation of coordinates that $J_0:=[[0,1], [2,3], \dots, [n, n+1]]$ is an element of $\KKK$.
Consider the $(d+1)$-dimensional
subspace of $\P^{n+1}$ defined by
$$
z_2-\eta z_3=z_4-\eta z_5= \cdots = z_{2d}-\eta z_{2d+1} =0.
$$
Then its intersection with $X$ is defined in $\P^{n+1}$ by
$$
z_0^m+z_1^m=z_2-\eta z_3=z_4-\eta z_5= \cdots = z_{n}-\eta z_{n+1} =0,
$$
which  is the union of $m$ standard
$d$-spaces
$L_{[J_0, (\eta\zeta^{\nu}, \eta, \dots, \eta)]}$
for $\nu=0, \dots, m-1$ in $X$.
Thus we have $P_X\in \LLLKKK(X)$ and Claim~\ref{claim:LLLKKK} is proved.
\end{proof}
Since $\LLLKKK\sprime(X)$ is an $R$-submodule of the ideal $V_n(X)\dual$ of $R$
and $R/V_n(X)\dual$ is torsion free,
the torsion of $H_n(X)/\LLLKKK(X)\cong V_n(X) \dual/\LLLKKK\sprime(X)$
is isomorphic to the torsion of $R/\LLLKKK\sprime(X)$.
Therefore, in order to prove Part (a) of Theorem~\ref{thm:main},
it is enough to show that the ideal $\LLLKKK\sprime(X)$ of $R$ is generated by
the polynomials $\psi_J$,
where $J$ runs through $\KKK$.
\par
For each $J=[[j_0,k_0], \dots, [j_d, k_d]] \in \JJJ$,
we let $G$ acts on the set $\BBB$ by
\begin{equation}\label{eq:gJ}
[g, J](\beta):=
(
\;
\zeta^{-\nu_{k_0}}\beta_0,
\;\; \zeta^{\nu_{j_1}-\nu_{k_1}}\beta_1,
\dots,
\;\;
\zeta^{\nu_{j_d}-\nu_{k_d}}\beta_d
\;).
\end{equation}
Then we have
$$
g\inv (L_{J, \beta})=L_{J, [g, J](\beta)}.
$$
Moreover,
for any $\beta, \beta\sprime\in \BBB$ and $J\in \JJJ$,
there exists $g\in G$ such that $\beta\sprime=[g, J](\beta)$.
Hence,
for a fixed $J\in \JJJ$,
the $\Z$-submodule $\LLL_{\{J\}}(X)$
of $H_n(X)$
generated by the classes $[L_{J, \beta}]$
of $L_{J, \beta}$ ($\beta\in \BBB$)
is the
$R$-submodule generated by a single element
$[L_{J, (\eta, \dots, \eta)}]$.
It is therefore enough to show that
the image $\psi_J\sprime$ of $[L_{J, (\eta, \dots, \eta)}]$
by the homomorphism $\LLLKKK(X) \inj H_n(X) \surj V_n(X) \dual \inj R$ is equal to $\psi_J$
up to sign.
\par
Suppose that
$$
\psi_J\sprime=\sum a_{\nu_1\dots\nu_{n+1}} t_1^{\nu_1}\cdots t_{n+1}^{\nu_{n+1}},
$$
where the summation is taken over all $(n+1)$-tuples $(\nu_1, \dots, \nu_{n+1})\in (\Z/m\Z)^{n+1}$,
and $a_{\nu_1\dots\nu_{n+1}}\in \Z$.
For simplicity, we put
$$
e(\nu):=s(\zeta^{-\nu} \eta)= \begin{cases}
1 & \textrm{if $\nu=0$}, \\
-1 & \textrm{if $\nu=1$}, \\
0 & \textrm{otherwise}.
\end{cases}
$$
Then,
writing $\gamma_1^{\nu_1}\cdots \gamma_{n+1}^{\nu_{n+1}}$ by $g$,
we have
\begin{eqnarray*}
a_{\nu_1\dots\nu_{n+1}}
&=& \intp{L_{J, (\eta, \dots, \eta)}, g(S)} \\
&= & \intp{g\inv(L_{J, (\eta, \dots, \eta)}), S} \\
&= & \intp{L_{J, [g, J](\eta, \dots, \eta)}, S} \\
&= & \sgn(\sigma_J) e(\nu_{k_0}) e({\nu_{k_1}-\nu_{j_1}}) \cdots
e({\nu_{k_d}-\nu_{j_d}}).
\end{eqnarray*}
where the last equality follows from Theorem~\ref{thm:LS}.
It remains to notice that
$$
\sum_{\nu\in \Z/m\Z} e({\nu})t^{\nu}=1-t
\quand
\sum_{\nu, \nu\sprime\in \Z/m\Z} e(\nu-\nu\sprime)t_1^{\nu}t_2^{\nu\sprime}=
(1-t_1)\phi(t_1t_2).
$$
Therefore we
do
have $\psi_J\sprime=\pm \psi_J$.
\qed
\subsection{Proof of Theorem~\ref{thm:mainrank} and Corollary~\ref{cor:primem}}\label{subsec:proofrank}
We put
$$
A_{\KKK}:= R/\ideal(\psi_J | J \in \KKK).
$$
Let $K_p$ be an algebraically closed field of characteristic $p\ge 0$.
Since
$$
\dim_{K_p} (R\tensor K_p)=m^{n+1}
$$
does not depend on $p$,
the $\Z$-module $A_{\KKK}$ has a torsion element of order $p$ if and only if
$$
\dim_{K_p} (A_{\KKK}\tensor K_p) > \dim_{\C} (A_{\KKK}\tensor \C).
$$
On the other hand,
by Claim~\ref{claim:LLLKKK} and $\LLLKKK\sprime(X)=\ideal(\psi_J | J \in \KKK)$ in $R$,
we have
$$
\rank \LLLKKK(X)=m^{n+1}- \dim_{\C} (A_{\KKK}\tensor \C) +1.
$$
Therefore it is enough to prove the following:
\begin{claim}\label{claim:Kp}
If $p=0$ or $(p, m)=1$, then
$$
\dim_{K_p} (A_{\KKK}\tensor K_p)=m^{n+1}-|\Gamma_{\KKK}|.
$$
\end{claim}
Thus, from
now on we assume that $p=0$ or $(p, m)=1$.
Then $R\tensor K_p$ is a semisimple ring, and all its simple modules have
dimension one over~$K_p$: they correspond to the multi-eigenvalues of
$(t_1,\ldots,t_{n+1})$, which are all $m$-th roots of unity (\latin{cf}.
Definition~\ref{def:gamma} in the case $K_p=\C$).
In other words,
$$
M:=\Spec (R\tensor K_p)
$$
is a reduced
scheme
of dimension zero
consisting of $m^{n+1}$ closed points.
Then $\Spec (A_{\KKK}\tensor K_p)$ is a closed subscheme $M_{\KKK}$ of $M$,
and $\dim_{K_p} (A_{\KKK}\tensor K_p)$ is the number of closed points of $M_{\KKK}$.
Let $\Gamma_{\KKK}$ be the subset of $M$ defined by  Definition~\ref{def:gamma}
with $\C$ replaced by $K_p$.
Note that, for $a\in K_p\sptimes$ with $a^m=1$, we have
$$
\phim(a)=0\;\; \Longleftrightarrow\;\; a\ne 1.
$$
Therefore, for $P=(a_1, \dots, a_{n+1})\in M$, we have
\begin{eqnarray*}
P\notin M_{\KKK} & \Longleftrightarrow & \psi_J(a_1, \dots, a_{n+1})\ne 0\;\; \textrm{for some $J\in \KKK$} \\
& \Longleftrightarrow & a_{k_0}\ne 1,\dots, a_{k_d}\ne 1 \;\textrm{and}\; a_{j_1}a_{k_1}=\cdots= a_{j_d}a_{k_d}=1 \\
&& \phantom{aaaaaaaaaaaaaaaa} \textrm{for some $J=[[j_0, k_0], \dots, [j_d, k_d]]\in \KKK$} \\
& \Longleftrightarrow & a_{i}\ne 1 \;\; \textrm{for $i=1, \dots, n+1$} \quand  \\
&& a_{j_1}a_{k_1}=\cdots= a_{j_d}a_{k_d}=1 \;\;\textrm{for some $J=[[j_0, k_0], \dots, [j_d, k_d]]\in \KKK$} \\
& \Longleftrightarrow &  P\in \Gamma_{\KKK}.
\end{eqnarray*}
Therefore we have $\dim_{K_p} (A_{\KKK}\tensor K_p)=|M_{\KKK}|=|M|-|\Gamma_{\KKK}|$.
This
concludes the proof of Claim~\ref{claim:Kp} and, hence, that
of Theorem~\ref{thm:mainrank} and Corollary~\ref{cor:primem}.
\qed
\begin{remark}\label{rem:dim}
The rank of $\LLL(X)=\LLL_{\JJJ}(X)=1+|\Gamma_{\JJJ}|$ is equal to the constant term of  the expansion of
\begin{equation*}\label{eq:Xpol}
\begin{cases}
1+(x_1+\cdots+x_{h-1}+1+x_{h-1}\inv+\cdots +x_1\inv )^{n+2} & \textrm{if $m=2h$ is even},\\
1+(x_1+\cdots+x_h+x_h\inv+\cdots +x_1\inv )^{n+2} & \textrm{if $m=2h+1$ is odd.} \\
\end{cases}
\end{equation*}
For small dimensions~$n$, we have
\begin{equation*}
\rank \LLL(X)=
\begin{cases}
3 m^2-9 m+6 +\delta_{m} & \textrm{ for $n=2$,}\\
15 m^3-90 m^2+175 m -100  +(15m-39)\delta_{m} & \textrm{ for $n=4$,}\\
105 m^4-1050 m^3+3955 m^2-6335 m +3325  +{}&\\
\phantom{aaaaaaaaaaaa} +(210m^2-1302m+2010)\delta_{m} & \textrm{ for $n=6$,}
\end{cases}
\end{equation*}
where $\delta_{m}\in \{0, 1\}$ satisfies $\delta_{m}\equiv m-1\bmod 2$.
\end{remark}
\subsection{Proof of Part (b) of Theorem~\ref{thm:main}}\label{subsec:proofb}
The following lemma is immediate:
\begin{lemma}\label{lem:phirho}
In $\Z[x, y]/(x^m-1, y^m-1)$, we have
$$
(y-1)\phim(xy)=-(x-1)(y-1)\rhom(x, y).
$$
\end{lemma}
We put
$$
\lambda:=(t_1-1)\cdots (t_{n+1}-1).
$$
By Lemma~\ref{lem:phirho},  we have
$$
\psi_J:=\pm \lambda \rhom_J.
$$
Hence $R/\ideal(\psi_J | J \in \KKK)$ in Part (a) of Theorem~\ref{thm:main} is equal to
$R/\ideal(\lambda \rhom_J | J \in \KKK)$.
Consider the natural exact sequence
$$
0
\;\;\to\;\;
(\lambda)/\ideal(\lambda \rhom_J | J \in \KKK)
\;\;\to\;\;
R/\ideal(\lambda \rhom_J | J \in \KKK)
\;\;\to\;\;
R/(\lambda)
\;\;\to\;\;
0.
$$
Since $R/(\lambda)$ is a free $\Z$-module by Lemma~\ref{lem:A},
the torsion of $R/\ideal(\psi_J | J \in \KKK)$ is isomorphic to the torsion of
$(\lambda)/\ideal(\lambda \rhom_J | J \in \KKK) $.
The homomorphism $R\surj (\lambda)$ given by
$f\mapsto f\lambda$ identifies $(\lambda)$ with $\barR$  by
Lemma~\ref{lem:A},
and under this identification,
the submodule $\ideal(\lambda \rhom_J | J \in \KKK) $ of $(\lambda)$
coincides with the ideal $\ideal(\rhom_J | J \in \KKK)$ of $\barR$.
Therefore we have $(\lambda)/\ideal(\lambda \rhom_J | J \in \KKK) \cong \barR/\ideal(\rhom_J | J \in \KKK)$.
\qed
\subsection{Proof of Parts (c) and (d) of Theorem~\ref{thm:main}}\label{subsec:proofcd}
%
%
%\par
%
Part (c) and Part (d) are dual to Part (a) and Part (b), respectively.
We use the following simple observation.
Let $\varphi\colon M_1\to M_2$ be a homomorphism
of free $\Z$-modules,
and let
$\varphi\dual\colon M_2\dual\to M_1\dual$ be the dual of $\varphi$.
Then there exist canonical isomorphisms
$$
\Tors\Coker(\varphi)=
 \operatorname{Ext}_{\Z}(\Tors\Coker(\varphi\dual),\Z)=
 \Hom_{\Z}(\Tors\Coker(\varphi\dual),\Q/\Z),
$$
where $\Tors M$ denotes the torsion of a $\Z$-module $M$.
Hence, there also exists a non-canonical  isomorphism
$\Tors\Coker(\varphi)\cong\Tors\Coker(\varphi\dual)$.
\par
\medskip
We put
$$
L_{\KKK}:=\bigcup_{J\in \KKK,\; \beta\in \BBB} L_{J, \beta},
$$
and consider the groups
$$
H_n(L_{\KKK})=\bigoplus_{J\in \KKK,\; \beta\in \BBB}\Z [L_{J, \beta}],
\quad
H^n(L_{\KKK})=\bigoplus_{J\in \KKK,\; \beta\in \BBB}\Z [L_{J, \beta}]\dual,
$$
each of which has a natural structure of the $R$-modules~(see~\eqref{eq:gJ}).
The inclusion $L_{\KKK}\inj X$ induces an $R$-linear homomorphism
$$
\varphi\colon H_n (L_{\KKK}) \to H_n(X).
$$
Then $H_n(X)/\LLL_{\KKK}(X)=\Coker(\varphi)$.
Note that $\intpempty$ defines an isomorphism $H_n(X)\cong H_n(X)\dual$
(the Poincar\'{e} duality),
and hence we obtain the dual homomorphism
$$
\varphi\dual\colon  H_n(X)\to H^n(L_{\KKK}).
$$
By the observation above, the torsion in question is the dual of the torsion of $\Coker(\varphi\dual)$,
and hence these torsions are isomorphic.
Consider the composite
$$
\varphi\dual_V\colon  R\surj V_n(X)\inj  H_n(X)\to H^n(L_{\KKK}),
$$
where the first surjection is given by
Theorem~\ref{thm:pham}.
Since $V_n(X)$ is primitive in $H_n (X)$ (see~\eqref{eq:longexact}),
the torsion of $H_n(X)/\LLL_{\KKK}(X)$ is isomorphic to the torsion of $\Coker (\varphi\dual_V)$.
Recall that we regard $H^n(L_{\KKK})$ as an $R$-module
\latin{via}
$$
g([L_{J, \beta}]\dual)=[L_{J, [g\inv, J]\beta}]\dual.
$$
For $J=[[j_0, k_0], \dots, [j_d, k_d]]\in \KKK$,
 the natural homomorphism
\begin{equation}\label{eq:RRe}
R\surj R [L_{J, (\eta, \dots, \eta)}]\dual=\bigoplus_{\beta\in \BBB} \Z [L_{J,\beta}]\dual
\end{equation}
given by $1\mapsto [L_{J, (\eta, \dots, \eta)}]\dual$
identifies $R [L_{J, (\eta, \dots, \eta)}]\dual$ with
\begin{equation}\label{eq:RJ}
R_J= R/ (t_{j_1}t_{k_1}-1, \dots, t_{j_d}t_{k_d}-1)=\Z[t_{k_0}, \dots, t_{k_d}]/(t_{k_0}^m-1, \dots, t_{k_d}^m-1),
\end{equation}
where the second equality follows from the relations $t_{j_{\nu}}=t_{k_{\nu}}^{m-1}$ ($\nu=1, \dots, d$) in $R_J$.
Indeed, each $t_{j_{i}} t_{k_{i}}-1$
is contained in the kernel of~\eqref{eq:RRe}
by the definition~\eqref{eq:gJ} of the action of $G$,
and both $\Z$-modules   $R_J$ and $R [L_{J, (\eta, \dots, \eta)}]\dual$ are free of rank $m^{d+1}=|\BBB|$.
Hence
we have
$$
H^n(L_{\KKK})=\bigoplus_{J\in \KKK} R_J.
$$
The homomorphism
$\varphi\dual_V$ is given by
$$
1\mapsto \sum_{J\in \KKK}\sum_{\beta\in \BBB} \intp{S, L_{J, \beta}} [L_{J,\beta}]\dual.
$$
For  $J=[[j_0, k_0], \dots, [j_d, k_d]]\in \KKK$,
we have
$$
[(\gamma_{k_0}^{-\alpha_0} \cdots \gamma_{k_d}^{-\alpha_d})\inv ,J](\eta, \dots, \eta)=(\zeta^{-\alpha_0}\eta, \dots, \zeta^{-\alpha_d}\eta),
$$
and hence,
by Theorem~\ref{thm:LS}, we obtain
\begin{eqnarray*}
&& \sum_{\beta\in \BBB} \intp{S, L_{J, \beta}} [L_{J,\beta}]\dual\\
 &=&
 \sgn(\sigma_J) \sum_{\alpha_0\in \Z/m\Z}\cdots \sum_{\alpha_d\in \Z/m\Z}
 e(\alpha_0)\cdots e(\alpha_d)  [L_{J,(\zeta^{-\alpha_0}\eta, \dots, \zeta^{-\alpha_d}\eta)}]\dual \\
 &=&
 \sgn(\sigma_J) \sum_{\alpha_0=0}^1\cdots \sum_{\alpha_d=0}^1
 e(\alpha_0)\cdots e(\alpha_d)  \gamma_{k_0}^{-\alpha_0} \cdots \gamma_{k_d}^{-\alpha_d} [L_{J,(\eta, \dots,\eta)}]\dual \\
  &=&
  \sgn(\sigma_J)  (1-t_{k_0}^{-1})\cdots (1-t_{k_d}^{-1})  [L_{J,(\eta, \dots,\eta)}]\dual\\
  &=&
  \sgn(\sigma_J)   (t_{k_0}-1) \cdots   (t_{k_d}-1) t_{k_0}\inv \cdots t_{k_d}\inv [L_{J,(\eta, \dots,\eta)}]\dual\\
   &=&    \tau_J c_J,
\end{eqnarray*}
where $c_J:=\sgn(\sigma_J)t_{k_0}\inv \cdots t_{k_d}\inv [L_{J,(\eta, \dots,\eta)}]\dual$.
Note that $\sgn(\sigma_J)  t_{k_0}\inv \cdots t_{k_d}\inv$ is a unit in $R_J$.
Replacing the generator $[L_{J,(\eta, \dots,\eta)}]\dual$  of  each factor of
$H^n(\LLL_{\KKK})=\bigoplus_{J\in \KKK} R_J$
by $c_J$,
the image of $\varphi\dual_V$ is the $R$-submodule $\MMM$ generated  by
$$
s:=\sum_{J\in \KKK}  \tau_J 1_J.
$$
Thus Part (c) is proved.
\par
For  $J=[[j_0, k_0], \dots, [j_d, k_d]]\in \KKK$,
let $(\tau_J)$ be the ideal of $R_J$ generated by $\tau_J$.
Then $s \in L_{\KKK}\dual=\bigoplus_{J\in \KKK} R_J$ is contained in $\bigoplus_{J\in \KKK} (\tau_J)$.
We consider the exact sequence
$$
0
\;\;\to\;\;
\left(\bigoplus_{J\in \KKK}  (\tau_J)\right)/Rs
\;\;\to\;\;
\left(\bigoplus_{J\in \KKK} R_J\right)/Rs
\;\;\to\;\;
\bigoplus_{J\in \KKK} \left(R_J/(\tau_J)\right)
\;\;\to\;\;
0.
$$
Since
$$
R_J/ (\tau_J)=\Z[t_{k_0}, \dots, t_{k_d}]/(t_{k_0}^m-1, \dots, t_{k_d}^m-1, \tau_J)
$$
is a free $\Z$-module by the second equality of~\eqref{eq:RJ} and Lemma~\ref{lem:A},
the torsion of $\bigoplus_{J\in \KKK} R_J/Rs$ is isomorphic to
the torsion of $\bigoplus_{J\in \KKK}  (\tau_J)/Rs$.
On the other hand,
the homomorphism  $R_J\surj (\tau_J)$ given by $f\mapsto f\tau_J$
identifies $(\tau_J)$ with
$$
\barR_J=\Z[t_{k_0}, \dots, t_{k_d}]/(\phim(t_{k_0}), \dots, \phim(t_{k_d}))
$$
by Lemma~\ref{lem:A},
and under this identification, the element $\tau_J\in (\tau_J)$ corresponds to the multiplicative unit $1_J$ of
$\barR_J$.
Therefore, by $\bigoplus_{J\in \KKK}  (\tau_J)\cong \bigoplus_{J\in \KKK} \barR_J$,
the element $s\in \bigoplus_{J\in \KKK} R_J$ corresponds to $\sum_{J\in\KKK} 1_J \in \bigoplus_{J\in \KKK} \barR_J$.
Hence $(\bigoplus_{J\in \KKK}  (\tau_J))/Rs$ is isomorphic to
$(\bigoplus_{J\in \KKK} \barR_J)/ \widebar{\MMM}$.
\qed
\subsection{Proof of Corollaries~\ref{cor:W} and~\ref{cor:Wext}}\label{subsec:proofcorW}
To prove Corollary~\ref{cor:W},
we merely put
$J_0:=[[0,1],[2,3], \dots, [n, n+1]]$,
and apply Part (d) of Theorem~\ref{thm:main} to the case
$\KKK=\{ J_0\}$.
We immediately see that $\LLL_{\{ J_0\}}(X)$ is primitive in $H_n(X)$.
Let $\WWW=\{W_t\}_{t\in U}$ be the family of smooth hypersurfaces defined by the equations of the form~\eqref{eq:eqW}.
The parameter space $U$ of this family is connected,
and hence there exists a path $\gamma\colon [0, 1]\to U$ from
the Fermat variety $X=W_{\gamma(0)}$  to an arbitrary member $W=W_{\gamma(1)}$ of $\WWW$.
Along the family $W_{\gamma(t)}$, the
subspaces $L_{J_0, \beta}$ ($\beta\in \BBB$) in $X$ deform
to
subspaces of $W_{\gamma(t)}$
defined by equations of the form
$$
\beta^{(i)}_{\nu}(t) z_{2i} =\alpha^{(i)}_{\nu}(t) z_{2i+1}\quad (i=0, \dots, d,\;\; \nu=1, \dots, m).
$$
Thus, along the constant
(with respect to the Gauss--Manin connection)
family $H_n(W_{\gamma(t)})$ of $\Z$-modules over $\gamma$,
the submodule $\LLL_{\{J_0\}}(X)$ of $H_n(X)$ is transported to the submodule of $H_n(W)$
generated by  the classes $[L\sprime_{(\nu_0, \dots, \nu_{d})}]$
of
subspaces $L\sprime_{(\nu_0, \dots, \nu_{d})}$ in $W$.
The rank and the primitivity are preserved during the
transport.
\par
For Corollary~\ref{cor:Wext}, we use the same continuity argument,
deforming $W_{s}$ to the Fermat variety and
representing the submodule in question as $\LLL_{\JJJ_s}(X)$, where $\JJJ_s$ is
the set of all partitions ``identical beyond~$s$'', \ie, those of the form
$$
[[j_0, k_0], \dots, [j_s,k_s], [2s+2,2s+3],\dots,[n, n+1]],\quad
0\le j_i,k_i\le 2s+1.
$$
The restriction of $\JJJ_s$ to the index set $\overline{2s+1}$ is well-defined and coincides with the full set
$\JJJ(2s)$ of partitions of $\overline{2s+1}$.
Then, denoting by $(\,\cdot\,)$ the dependence on the dimension (or
the number of
variables in the polynomial rings), it is easy to see that the module
$\bmodule_{\JJJ_s}(2d)$ given by Part~(d) of Theorem~\ref{thm:main}
can be represented in the form
$$
\bmodule_{\JJJ_s}(2d)=\bmodule_{\JJJ(2s)}(2s)\otimes_{\Z}\widebar S(s,d),
$$
where
$$
\widebar S(s,d):=\Z[t_{2s+2},t_{2s+4},\ldots,t_{2d}]/
 ( \phim(t_{2s+2}), \phim(t_{2s+4}), \dots, \phim(t_{2d})).
$$
(Since the tail of each partition is fixed, we have the
``constant'' relations
$$
t_{2s+2}t_{2s+3}=\dots=t_{2d}t_{2d+1}=1;
$$
hence, we can retain the even index
variables only and take these variables out.)
Thus, this module is free (as an abelian group) if and only if so is
$\bmodule_{\JJJ(2s)}(2s)$, \ie, if and only if
Conjecture~\ref{conj:main} holds for Fermat varieties of dimension~$2s$
in $\P^{2s+1}$.
\par
For the last assertion of Corollary~\ref{cor:Wext}, we observe that
Conjecture~\ref{conj:main} does hold for the Fermat varieties of dimension~$0$
(obvious) and~$2$ (see~\cite{arXiv13053073}).
\qed
\section{Computational criterion}\label{sec:comp}
In this section,
we focus on the description of the torsion of $H_n(X)/\LLLKKK(X)$
given by Part~(b) of Theorem~\ref{thm:main}.
We put
$$
B_{\KKK}:=\barR/\ideal(\rhom_J| J\in \KKK).
$$
By Lemma~\ref{lem:phirho}, the ideal $\ideal(\rhom_J| J\in \KKK)$ defines
the closed subscheme $\Gamma_{\KKK}$ in the reduced $0$-dimensional scheme
$\Spec(\barR\tensor\C)=(\mu_m\setminus\{1\})^{n+1}$,
and hence we can calculate $d_0:=\dim_{\C} (B_{\KKK}\tensor\C)=|\Gamma_{\KKK}|$.
On the other hand, for each prime divisor $p$ of $m$,
we can calculate $d_p:=\dim_{\F_p}  (B_{\KKK}\tensor\F_p)$
by calculating a Gr\"{o}bner basis of the ideal
\begin{equation}\label{eq:idealKKK}
(\phim(t_1), \dots, \phim(t_{n+1}))+\ideal(\rhom_J|J\in \KKK)
\end{equation}
in the polynomial ring $\F_p[t_1, \dots, t_{n+1}]$.
By  Corollary~\ref{cor:primem}, we see that $\LLLKKK(X)$ is primitive in $H_n(X)$
if and only if
$d_0=d_p$ holds for any prime divisor $p$ of $m$.
\par
Using  this method,
we have confirmed the primitivity of $\LLL(X)=\LLL_{\JJJ}(X)$ in $H_n(X)$
by the computer-aided calculation in the following cases:
\begin{equation*}\label{eq:confirmed}
(n, m)=(4,m) \;\;\textrm{where}\;\; 3\le m\le 12,
\quad
(6, 3),(6,4),(6,5), (8,3).
\end{equation*}
\bibliographystyle{plain}

\begin{thebibliography}{10}

\bibitem{MR717587}
Noboru Aoki and Tetsuji Shioda.
\newblock Generators of the {N}\'{e}ron-{S}everi group of a {F}ermat surface.
\newblock In {\em Arithmetic and geometry, Vol. I}, volume~35 of {\em Progr.
  Math.}, pages 1--12. Birkh\"{a}user Boston, Boston, MA, 1983.

\bibitem{MR1417938}
David Cox, John Little, and Donal O'Shea.
\newblock {\em Ideals, varieties, and algorithms}.
\newblock Undergraduate Texts in Mathematics. Springer-Verlag, New York, second
  edition, 1997.
\newblock An introduction to computational algebraic geometry and commutative
  algebra.

\bibitem{arXiv13053073}
Alex Degtyarev.
\newblock Lines generate the {P}icard group of a {F}ermat surface.
\newblock {\em J. Number Theory}, 147:454--477, 2015.

\bibitem{arXiv13070382}
Alex Degtyarev.
\newblock On the {P}icard group of a {D}elsarte surface, 2013.
\newblock preprint, arXiv:1307.0382.

\bibitem{MR0195868}
Fr\'{e}d\'{e}ric Pham.
\newblock Formules de {P}icard-{L}efschetz g\'{e}n\'{e}ralis\'{e}es et ramification
  des int\'{e}grales.
\newblock {\em Bull. Soc. Math. France}, 93:333--367, 1965.

\bibitem{MR594486}
Ziv Ran.
\newblock Cycles on {F}ermat hypersurfaces.
\newblock {\em Compositio Math.}, 42(1):121--142, 1980/81.

\bibitem{MR2653207}
Matthias Sch\"{u}tt, Tetsuji Shioda, and Ronald van Luijk.
\newblock Lines on {F}ermat surfaces.
\newblock {\em J. Number Theory}, 130(9):1939--1963, 2010.

\bibitem{MR1794260}
Ichiro Shimada.
\newblock Lattices of algebraic cycles on {F}ermat varieties in positive
  characteristics.
\newblock {\em Proc. London Math. Soc. (3)}, 82(1):131--172, 2001.

\bibitem{MR2789841}
Ichiro Shimada and Nobuyoshi Takahashi.
\newblock Primitivity of sublattices generated by classes of curves on an
  algebraic surface.
\newblock {\em Comment. Math. Univ. St. Pauli}, 59(2):77--95, 2010.

\bibitem{MR552586}
Tetsuji Shioda.
\newblock The {H}odge conjecture for {F}ermat varieties.
\newblock {\em Math. Ann.}, 245(2):175--184, 1979.

\bibitem{MR833362}
Tetsuji Shioda.
\newblock An explicit algorithm for computing the {P}icard number of certain
  algebraic surfaces.
\newblock {\em Amer. J. Math.}, 108(2):415--432, 1986.

\bibitem{MR526513}
Tetsuji Shioda and Toshiyuki Katsura.
\newblock On {F}ermat varieties.
\newblock {\em T\^{o}hoku Math. J. (2)}, 31(1):97--115, 1979.

\end{thebibliography}

\def\cftil#1{\ifmmode\setbox7\hbox{$\accent"5E#1$}\else
  \setbox7\hbox{\accent"5E#1}\penalty 10000\relax\fi\raise 1\ht7
  \hbox{\lower1.15ex\hbox to 1\wd7{\hss\accent"7E\hss}}\penalty 10000
  \hskip-1\wd7\penalty 10000\box7} \def\cprime{$'$} \def\cprime{$'$}
  \def\cprime{$'$} \def\cprime{$'$}

\end{document}